\documentclass{amsart}

\usepackage{amsmath}
\usepackage{amsthm}
\usepackage{amsfonts}
\usepackage{dsfont}

\newcommand{\reals}{\mathbb{R}}

\newcommand{\bracketb}[1]{\Big[#1\Big]}

\newcommand{\pb}[1]{\left\{#1\right\}}
\newcommand{\pba}[1]{\big\{#1\big\}}

\newcommand{\norm}[1]{\left|\left|#1\right|\right|}

\newcommand{\paraa}[1]{\big(#1\big)}
\newcommand{\parab}[1]{\Big(#1\Big)}
\newcommand{\parac}[1]{\bigg(#1\bigg)}

\newtheorem{theorem}{Theorem}[section]

\newtheorem{lemma}[theorem]{Lemma}
\newtheorem{proposition}[theorem]{Proposition}

\theoremstyle{definition}
\newtheorem{definition}[theorem]{Definition}
\theoremstyle{remark}
\newtheorem{remark}[theorem]{Remark}
\numberwithin{equation}{section}

\newcommand{\tr}{\operatorname{tr}}
\newcommand{\Tr}{\operatorname{Tr}}
\newcommand{\Z}{\mathcal{Z}}

\newcommand{\B}{\mathcal{B}}
\renewcommand{\P}{\mathcal{P}}
\renewcommand{\S}{\mathcal{S}}
\newcommand{\T}{\mathcal{T}}
\newcommand{\C}{\mathcal{C}}
\newcommand{\W}{\mathcal{W}}
\newcommand{\R}{\mathcal{R}}
\newcommand{\J}{\mathcal{J}}
\newcommand{\JM}{\mathcal{J}_M}

\newcommand{\xv}{\vec{x}}
\newcommand{\gb}{\,\bar{\!g}}
\renewcommand{\d}{\partial}
\newcommand{\TSigma}{T\Sigma}
\newcommand{\eps}{\varepsilon}

\newcommand{\nablab}{\bar{\nabla}}

\newcommand{\Gammab}{\bar{\Gamma}}
\newcommand{\Rb}{\bar{R}}
\newcommand{\TN}{T^{(N)}}

\newcommand{\Nh}{\hat{N}}

\newcommand{\av}{\vec{a}}
\newcommand{\cv}{\vec{c}}
\newcommand{\nAv}{\vec{n}_A}
\renewcommand{\mid}{\mathds{1}}

\title[Geometry of embedded manifolds in terms of Nambu brackets]{On the classical geometry of embedded\\ manifolds in terms of Nambu brackets}

\author{Joakim Arnlind}
\address[Joakim Arnlind]{Max Planck Institute for Gravitational Physics\\ 
Am M\"uhlenberg 1\\
D-14476 Golm\\
Germany}
\email{joakim.arnlind@aei.mpg.de}

\author{Jens Hoppe}
\address[Jens Hoppe]{Department of Mathematics\\
KTH\\
S-10044 Stockholm\\
Sweden}
\email{hoppe@math.kth.se}

\author{Gerhard Huisken}
\address[Gerhard Huisken]{Max Planck Institute for Gravitational Physics\\
Am M\"uhlenberg 1\\
D-14476 Golm\\
Germany}
\email{gerhard.huisken@aei.mpg.de}

\thanks{}

\subjclass[2000]{}
\keywords{}

\begin{document}

\begin{abstract}
  We prove that many aspects of the differential geometry of embedded
  Riemannian manifolds can be formulated in terms of a multi-linear
  algebraic structure on the space of smooth functions.  In
  particular, we find algebraic expressions for Weingarten's formula,
  the Ricci curvature and the Codazzi-Mainardi equations.
\end{abstract}

\maketitle

\section{Introduction}

\noindent Given a manifold $\Sigma$, it is interesting to study in
what ways information about the geometry of $\Sigma$ can be
extracted as algebraic properties of the algebra of smooth functions
$C^\infty(\Sigma)$. In case $\Sigma$ is a Poisson manifold, this
algebra has a second (apart from the commutative multiplication of
functions) bilinear (non-associative) algebra structure realized as the
Poisson bracket. The bracket is compatible with the
commutative multiplication via Leibniz rule, thus carrying the basic
properties of a derivation. 

On a surface $\Sigma$, with local coordinates $u^1$ and $u^2$, one may
define
\begin{align*}
  \{f,h\} = \frac{1}{\sqrt{g}}\parac{\frac{\d f}{\d u^1}\frac{\d h}{\d u^2}-
  \frac{\d h}{\d u^1}\frac{\d f}{\d u^2}},
\end{align*}
where $g$ is the determinant of the induced metric tensor, and one can readily
check that $\paraa{C^\infty(\Sigma),\{\cdot,\cdot\}}$ is a Poisson
algebra. Having only this very particular combination of derivatives
at hand, it seems at first unlikely that one can encode geometric
information of $\Sigma$ in Poisson algebraic
expressions. Surprisingly, it turns out that many differential
geometric quantities can be computed in a completely algebraic way,
cp.  Theorem \ref{thm:ricciCurvature} and Theorem
\ref{thm:CMNambu}. For instance, the Gaussian curvature of a surface
embedded in $\reals^3$ can be written as
\begin{align*}
  K = -\frac{1}{2}\sum_{i,j=1}^3\{x^i,n^j\}\{x^j,n^i\},
\end{align*}
where $x^i(u^1,u^2)$ are the embedding coordinates and $n^i(u^1,u^2)$
are the components of a unit normal vector at each point of $\Sigma$.

For a general $n$-dimensional manifold $\Sigma$, we are led to consider
Nambu brackets \cite{n:generalizedmech}, i.e. multi-linear $n$-ary maps from
$C^\infty(\Sigma)\times\cdots\times C^\infty(\Sigma)$ to $C^\infty(\Sigma)$, defined by
\begin{align*}
  \{f_1,\ldots,f_n\} = \frac{1}{\sqrt{g}}\eps^{a_1\cdots a_n}\paraa{\d_{a_1}f_1}\cdots\paraa{\d_{a_n} f_n}.
\end{align*}

Our initial motivation for studying this
problem came from matrix regularizations of Membrane Theory. Classical
solutions in Membrane Theory are 3-manifolds with vanishing mean
curvature in $\reals^{1,d}$. Considering one of the coordinates to be
time, the problem can also be formulated in a dynamical way as
surfaces sweeping out volumes of vanishing mean curvature. In this
context, a regularization was introduced replacing the infinite
dimensional function algebra on the surface by an algebra of $N\times
N$ matrices \cite{h:phdthesis}. If we let $\TN$ be a linear map from
smooth functions to hermitian $N\times N$ matrices, the regularization
is required to fulfill
\begin{align*}
  &\lim_{N\to\infty}\norm{\TN(f)\TN(g)-\TN(fg)}=0,\\
  &\lim_{N\to\infty}\norm{N[\TN(f),\TN(h)]-i\TN(\{f,h\})}=0,
\end{align*}
where $||\cdot||$ denotes the operator norm, and therefore it is
natural to regularize the system by replacing (commutative)
multiplication of functions by (non-commutative) multiplication of
matrices and Poisson brackets of functions by commutators of matrices.

Although we may very well consider $\TN(\frac{\d f}{\d u^1})$, its
relation to $\TN(f)$ is in general not simple. However, the particular
combination of derivatives in $\TN(\{f,h\})$ is expressed in terms of
a commutator of $\TN(f)$ and $\TN(h)$. In the context of Membrane
Theory, it is desirable to have geometrical quantities in a form that
can easily be regularized, which is the case for any expression
constructed out of multiplications and Poisson brackets.

The paper is organized as follows: In Section \ref{sec:preliminaries}
we introduce the relevant notation by recalling some basic facts about
submanifolds. In Section \ref{sec:nambuPoissonFormulation} we
formulate several basic differential geometric objects in terms of
Nambu brackets, and in Section \ref{sec:normalVectors} we
provide a construction of a set of orthonormal basis vectors of the
normal space. Section \ref{sec:CodazziMainardi} is devoted to the
study of the Codazzi-Mainardi equations and how one can rewrite them
in terms of Nambu brackets. Finally, in Section
\ref{sec:surfaces} we study the particular case of surfaces, for which
many of the introduced formulas and concepts are particularly nice and
in which case one can construct the complex structure in terms of
Poisson brackets.


\section{Preliminaries}\label{sec:preliminaries}

\noindent To introduce the relevant notations, we shall recall some
basic facts about submanifolds, in particular Gauss' and Weingarten's
equations (see
e.g. \cite{kn:foundationsDiffGeometryI,kn:foundationsDiffGeometryII}
for details). For $n\geq 2$, let $\Sigma$ be a $n$-dimensional manifold embedded in a
Riemannian manifold $M$ with $\dim M=n+p\equiv m$. Local coordinates on $M$ will be denoted by
$x^1,\ldots,x^m$, local coordinates on $\Sigma$ by $u^1,\ldots,u^n$,
and we regard $x^1,\ldots,x^m$ as being functions of $u^1,\ldots,u^n$
providing the embedding of $\Sigma$ in $M$. The metric tensor on $M$
is denoted by $\gb_{ij}$ and the induced metric on $\Sigma$ by
$g_{ab}$; indices $i,j,k,l,n$ run from $1$ to $m$, indices
$a,b,c,d,p,q$ run from $1$ to $n$ and indices $A,B,C,D$ run from $1$
to $p$. Furthermore, the covariant derivative and the Christoffel
symbols in $M$ will be denoted by $\nablab$ and $\Gammab^{i}_{jk}$
respectively.

The tangent space $\TSigma$ is regarded as a subspace of the tangent
space $TM$ and at each point of $\Sigma$ one can choose
$e_a=(\d_ax^i)\d_i$ as basis vectors in $\TSigma$, and in this basis
we define $g_{ab}=\gb(e_a,e_b)$. Moreover, we choose a set of normal
vectors $N_A$, for $A=1,\ldots,p$, such that
$\gb(N_A,N_B)=\delta_{AB}$ and $\gb(N_A,e_a)=0$.

The formulas of Gauss and Weingarten split the covariant derivative in
$M$ into tangential and normal components as
\begin{align}
  &\nablab_X Y = \nabla_X Y + \alpha(X,Y)\label{eq:GaussFormula}\\
  &\nablab_XN_A = -W_A(X) + D_XN_A\label{eq:WeingartenFormula}
\end{align}
where $X,Y\in \TSigma$ and $\nabla_X Y$, $W_A(X)\in\TSigma$ and
$\alpha(X,Y)$, $D_XN_A\in\TSigma^\perp$. By expanding $\alpha(X,Y)$ in
the basis $\{N_1,\ldots,N_p\}$ one can write (\ref{eq:GaussFormula}) as
\begin{align}
  &\nablab_X Y = \nabla_X Y + \sum_{A=1}^ph_A(X,Y)N_A,\label{eq:GaussFormulah}
\end{align}
and we set $h_{A,ab} = h_A(e_a,e_b)$. From the above equations one derives the relation
\begin{align}
  h_{A,ab} &= -\gb\paraa{e_a,\nablab_b N_A},
\end{align}
as well as Weingarten's equation
\begin{align}
  h_A(X,Y) = \gb\paraa{W_A(X),Y},  
\end{align}
which implies that $(W_A)^a_b = g^{ac}h_{A,cb}$, where $g^{ab}$
denotes the inverse of $g_{ab}$.  

From formulas (\ref{eq:GaussFormula}) and (\ref{eq:WeingartenFormula})
one obtains Gauss' equation, i.e. an expression for the curvature $R$
of $\Sigma$ in terms of the curvature $\Rb$ of $M$, as
\begin{equation}\label{eq:GaussEquation}
  \begin{split}
    g\paraa{R(X,Y)Z,V} =
    \gb&\paraa{\Rb(X,Y)Z,V}-\gb\paraa{\alpha(X,Z),\alpha(Y,V)}\\
    &+\gb\paraa{\alpha(Y,Z),\alpha(X,V)},    
  \end{split}
\end{equation}
where $X,Y,Z,V\in\TSigma$. As we shall later on consider the Ricci curvature,
let us note that (\ref{eq:GaussEquation}) implies
\begin{align}
  \R^p_b = g^{pd}g^{ac}\gb\paraa{\Rb(e_c,e_d)e_b,e_a}
  +\sum_{A=1}^p\bracketb{(W_A)^a_a(W_A)_b^p-(W_A^2)_b^p}
\end{align}
where $\R$ is the Ricci curvature of $\Sigma$ considered as a map
$\TSigma\to\TSigma$.  We also recall the mean curvature vector,
defined as
\begin{align}
  H = \frac{1}{n}\sum_{A=1}^p\paraa{\tr W_A}N_A.
\end{align}


\section{Algebraic formulation}\label{sec:nambuPoissonFormulation}

\noindent In this section we will prove that one can express many
aspects of the differential geometry of an embedded manifold $\Sigma$
in terms of a Nambu bracket introduced on $C^\infty(\Sigma)$.
Let $\rho:\Sigma\to\reals$ be an arbitrary non-vanishing
density and define
\begin{align}\label{eq:PbracketDef}
  \{f_1,\ldots,f_n\} = \frac{1}{\rho}\eps^{a_1\cdots a_n}\paraa{\d_{a_1}f_1}\cdots\paraa{\d_{a_n} f_n}
\end{align}
for all $f_1,\ldots,f_n\in C^\infty(\Sigma)$, where $\eps^{a_1\cdots
  a_n}$ is the totally antisymmetric Levi-Civita symbol with
$\eps^{12\cdots n}=1$. Together with this multi-linear map, $\Sigma$ is
a Nambu-Poisson manifold.


The above Nambu bracket arises from the choice of a
volume form on $\Sigma$. Namely, let $\omega$ be a volume form and
define $\{f_1,\ldots,f_n\}$ via the formula
\begin{align}\label{eq:nambuVolumeForm}
 \{f_1,\ldots,f_n\}\omega = df_1\wedge\cdots\wedge df_n.
\end{align}
Writing $\omega=\rho\, du^1\wedge\cdots\wedge du^n$ in local
coordinates, and evaluating both sides of (\ref{eq:nambuVolumeForm})
on the tangent vectors $\d_{u^1},\ldots,\d_{u^n}$ gives
\begin{align*}
  \{f_1,\ldots,f_n\} = \frac{1}{\rho}\det\parac{\frac{\d(f_1,\ldots,f_n)}{\d(u^1,\ldots,u^n)}}
  =\frac{1}{\rho}\eps^{a_1\cdots a_n}\paraa{\d_{a_1}f_1}\cdots\paraa{\d_{a_n}f_n}.
\end{align*}
To define the objects which we will consider, it is convenient to
introduce some notation. Let
$x^1(u^1,\ldots,u^n),\ldots,x^m(u^1,\ldots,u^n)$ be the embedding
coordinates of $\Sigma$ into $M$, and let $n_A^i(u^1,\ldots,u^n)$ denote the
components of the orthonormal vectors $N_A$, normal to $\TSigma$. Using
multi-indices $I=i_1\cdots i_{n-1}$ and $\av=a_1\cdots a_{n-1}$ we define
\begin{align*}
  &\{f,\xv^I\} \equiv \{f,x^{i_1},x^{i_2},\ldots,x^{i_{n-1}}\}\\
  &\{f,\nAv^I\} \equiv \{f,n_A^{i_1},n_A^{i_2},\ldots,n_A^{i_{n-1}}\},
\end{align*}
together with
\begin{align*}
  &\d_{\av}\xv^I \equiv \paraa{\d_{a_1}x^{i_1}}\paraa{\d_{a_2}x^{i_2}}\cdots\paraa{\d_{a_{n-1}}x^{i_{n-1}}}\\
  &\paraa{\nablab_{\av}\nAv}^I \equiv \paraa{\nablab_{a_1}N_A}^{i_1}\paraa{\nablab_{a_2}N_A}^{i_2}\cdots\paraa{\nablab_{a_{n-1}}N_A}^{i_{n-1}}\\
  &\gb_{IJ} \equiv \gb_{i_1j_1}\gb_{i_2j_2}\cdots\gb_{i_{n-1}j_{n-1}}.
\end{align*}
We now introduce the main objects of our study
\begin{align}
  \P^{iJ} &= \frac{1}{\sqrt{(n-1)!}}\{x^i,\xv^J\} = \frac{1}{\sqrt{(n-1)!}}\frac{\eps^{a\av}}{\rho}\paraa{\d_ax^i}\paraa{\d_{\av}\xv^J}\\
  \S_A^{iJ}&=\frac{(-1)^n}{\sqrt{(n-1)!}}\frac{\eps^{a\av}}{\rho}\paraa{\d_ax^i}\paraa{\nablab_{\av}\nAv}^J\\
  \T_A^{Ij} &=\frac{(-1)^n}{\sqrt{(n-1)!}}\frac{\eps^{\av a}}{\rho}\paraa{\d_{\av}\xv^I}\paraa{\nablab_aN_A}^j
\end{align}
from which we construct
\begin{align}
  \paraa{\P^2}^{ik} &= \P^{iI}\P^{kJ}\gb_{IJ}\\
  \paraa{\B_A}^{ik} &= \P^{iI}(\T_A)^{Jk}\gb_{IJ}\\
  \paraa{\S_A\T_A}^{ik} &=(\S_A)^{iI}(\T_A)^{Jk}\gb_{IJ}.
\end{align}
By lowering the second index with the metric $\gb$, we will also
consider $\P^2$, $\B_A$ and $\T_A\S_A$ as maps $TM\to TM$. Note that
both $\S_A$ and $\T_A$ can be written in terms of Nambu brackets, e.g.
\begin{align*}
  \T_A^{Ij} = \frac{(-1)^n}{\sqrt{(n-1)!}}\bracketb{\{\xv^I,n_A^j\}+\{\xv^I,x^k\}\Gammab^j_{kl}n_A^l}.
\end{align*}
Let us now investigate some properties of the maps defined above. As it will appear frequently, we define 
\begin{align}
  \gamma = \frac{\sqrt{g}}{\rho}.
\end{align}
It is useful to note that (cp. Proposition \ref{prop:TrPBST})
\begin{align*}
  \gamma^2 = \sum_{i,I=1}^m\frac{1}{n!}\{x^i,\xv^I\}\{\xv^I,x^i\},
\end{align*}
and to recall the cofactor expansion of the inverse of matrix:
\begin{lemma}
  Let $g^{ab}$ denote the inverse of $g_{ab}$ and $g=\det(g_{ab})$. Then
  \begin{align}
    gg^{ba} = \frac{1}{(n-1)!}\eps^{aa_1\cdots a_{n-1}}\eps^{bb_1\cdots b_{n-1}}g_{a_1b_1}g_{a_2b_2}\cdots g_{a_{n-1}b_{n-1}}.
  \end{align}
\end{lemma}

\begin{proposition}\label{prop:PBSTproperties}
  For $X\in TM$ it holds that
  \begin{align}
    &\P^2(X) = \gamma^2\gb(X,e_a)g^{ab}e_b\label{eq:P2X}\\
    &\B_A(X) = -\gamma^2\gb(X,\nablab_a N_A)g^{ab}e_b\label{eq:BAX}\\
    &\S_A\T_A(X) = \gamma^2(\det W_A)\gb(X,\nablab_aN_A)h_A^{ab}e_b,
  \end{align}
  and for $Y\in\TSigma$ one obtains
  \begin{align}
    &\P^2(Y) = \gamma^2Y\label{eq:P2Y}\\
    &\B_A(Y) = \gamma^2W_A(Y)\\
    &\S_A\T_A(Y) = -\gamma^2(\det W_A)Y.
  \end{align}
\end{proposition}

\begin{proof}
  Let us provide a proof for equations (\ref{eq:P2X}) and (\ref{eq:P2Y}); the other
  formulas can be proven analogously.
  \begin{align*}
    \P^2(X) &= \P^{iI}\P^{jJ}\gb_{IJ}\gb_{jk}X^k\d_i = 
    \frac{\eps^{a\av}\eps^{c\cv}}{\rho^2(n-1)!}\paraa{\d_ax^i}\paraa{\d_{\av}x^I}\paraa{\d_cx^j}\paraa{\d_{\cv}x^J}\gb_{IJ}\gb_{jk}X^k\d_i\\
    &= \frac{\eps^{a\av}\eps^{c\cv}}{\rho^2(n-1)!}g_{a_1c_1}\cdots g_{a_{n-1}c_{n-1}}\paraa{\d_ax^i}\paraa{\d_cx^j}\gb_{jk}X^k\d_i\\
    &= \gamma^2g^{ac}\paraa{\d_ax^i}\paraa{\d_cx^j}\gb_{jk}X^k\d_i
    = \gamma^2\gb(X,e_c)g^{ca}e_a.
  \end{align*}
  Choosing a tangent vector $Y=Y^ce_c$ gives immediately that $\P^2(Y)=\gamma^2Y$.
\end{proof}

\noindent For a map $\B:TM\to TM$ we denote the trace by $\Tr\B\equiv
\B^i_i$ and for a map $W:\TSigma\to\TSigma$ we denote the trace by $\tr W\equiv W^a_a$.

\begin{proposition}\label{prop:TrPBST}
  It holds that
  \begin{align}
    \frac{1}{n}\Tr\P^2 &= \gamma^2\\
    \Tr\B_A &= \gamma^2\tr W_A\\
    \frac{1}{n}\Tr\S_A\T_A &= -\gamma^2(\det W_A).
  \end{align}
\end{proposition}

\noindent A direct consequence of Propositions
\ref{prop:PBSTproperties} and \ref{prop:TrPBST} is that one can write the projection onto
$\TSigma$, as well as the mean curvature vector,  in terms of Nambu brackets.
\begin{proposition}
  The map 
  \begin{align}
    \gamma^{-2}\P^2=\frac{n}{\Tr\P^2}\P^2:TM\to\TSigma
  \end{align}
  is the orthogonal projection of $TM$ onto $\TSigma$. Furthermore,
  the mean curvature vector can be written as
\begin{align*}
  H = \frac{1}{\Tr\P^2}\sum_{A=1}^p\paraa{\Tr\B_A}N_A.
\end{align*}
\end{proposition}

\noindent Proposition \ref{prop:PBSTproperties} tells us that $\gamma^{-2}\B_A$
equals the Weingarten map $W_A$, when restricted to $\TSigma$. What is
the geometrical meaning of $\B_A$ acting on a normal vector? It turns
out that the maps $\B_A$ also provide information about the covariant
derivative in the normal space. If one defines $(D_X)_{AB}$ through
\begin{align*}
  D_XN_A = \sum_{B=1}^p(D_X)_{AB}N_B
\end{align*}
for $X\in\TSigma$, then one can prove the following relation to
the maps $\B_A$.
\begin{proposition}
  For $X\in\TSigma$ it holds that
  \begin{align}
    \gb\paraa{\B_B(N_A),X}=\gamma^2\paraa{D_X}_{AB}.
  \end{align}
\end{proposition}

\begin{proof}
  For a vector $X=X^ae_a$, it follows from Weingarten's formula (\ref{eq:WeingartenFormula}) that
  \begin{align*}
    (D_X)_{AB} = \gb\paraa{\nablab_X N_A,N_B}. 
  \end{align*}
  On the other hand, with the formula from Proposition
  \ref{prop:PBSTproperties}, one computes
  \begin{align*}
     \gb\paraa{\B_B(N_A),X} &= -\gamma^2\gb\paraa{N_A,\nablab_aN_B}g^{ab}g_{bc}X^c
     = -\gamma^2\gb\paraa{N_A,\nablab_XN_B}\\
     &= -\gamma^2(D_X)_{BA}=\gamma^2(D_X)_{AB}.
  \end{align*}
  The last equality is due to the fact that $D$ is a covariant
  derivative, which implies that
  $0=D_X\gb(N_A,N_B)=\gb(D_XN_A,N_B)+\gb(N_A,D_XN_B)$.
\end{proof}

\noindent Thus, one can write Weingarten's formula as
\begin{align}
  \gamma^2\nablab_XN_A = -\B_A(X)+\sum_{B=1}^p\gb\paraa{\B_B(N_A),X}N_B,
\end{align}
and since $h_A(X,Y) = \gamma^{-2}\gb(\B_A(X),Y)$ Gauss' formula becomes
\begin{align}
  \nablab_XY = \nabla_XY+\frac{1}{\gamma^2}\sum_{A=1}^p\gb\paraa{\B_A(X),Y}N_A.
\end{align}
Let us now turn our attention to the curvature of $\Sigma$. Since
Nambu brackets involve sums over all vectors in the basis of
$\TSigma$, one can not expect to find expressions for quantities that
involve a choice of tangent plane, e.g. the sectional curvature
(unless $\Sigma$ is a surface). However, it turns out that one can
write the Ricci curvature as an expression involving Nambu brackets.
\begin{theorem}\label{thm:ricciCurvature}
  Let $\R$ be the Ricci curvature of $\Sigma$, considered as a map
  $\TSigma\to\TSigma$. For any $X\in\TSigma$ it holds that
  \begin{align*}
    \R(X) = g^{pd}g^{ac}\gb\paraa{\Rb(e_c,e_d)e_b,e_a}X^be_p
    +\frac{1}{\gamma^4}\sum_{A=1}^p\bracketb{(\Tr\B_A)\B_A(X)-\B_A^2(X)},
  \end{align*}
  where $\Rb$ is the curvature tensor of $M$.
\end{theorem}

\begin{proof}
  The Ricci curvature of $\Sigma$ is defined as
  \begin{align*}
    \R^p_b = g^{ac}g^{pd}g\paraa{R(e_c,e_d)e_b,e_a}
  \end{align*}
  and from Gauss' equation (\ref{eq:GaussEquation}) it follows that
  \begin{align*}
    \R^p_b = g^{pd}g^{ac}\gb\paraa{\Rb(e_c,e_d)e_b,e_a}
    + g^{ac}g^{pd}\sum_{A=1}^p\parab{h_{A,bd}h_{A,ac}-h_{A,bc}h_{A,ad}}.
  \end{align*}
  Since $(W_A)^a_b = g^{ac}h_{A,cb}$ one obtains
  \begin{align*}
    \R_b^p = g^{ac}g^{pd}\gb\paraa{\Rb(e_c,e_d)e_b,e_a}
    + \sum_{A=1}^p\bracketb{\paraa{\tr W_A}(W_A)^p_b-(W_A^2)_b^p},
  \end{align*}
  and as $\B_A(X)=\gamma^2 W_A(X)$ for any $X\in\TSigma$, and
  $\Tr\B_A=\gamma^2\tr W_A$, one has
  \begin{equation*}
    \R(X) = g^{ac}g^{pd}\gb\paraa{\Rb(e_c,e_d)e_b,e_a}X^be_p
    + \frac{1}{\gamma^4}\sum_{A=1}^p\bracketb{\paraa{\Tr \B_A}\B_A(X)-\B_A^2(X)}.\qedhere
  \end{equation*}
\end{proof}

\section{Construction of normal vectors}\label{sec:normalVectors}

\noindent The results in Section \ref{sec:nambuPoissonFormulation}
involve Nambu brackets of the embedding coordinates and the
components of the normal vectors. In this section we will prove that
one can replace sums over normal vectors by sums of Nambu
brackets of the embedding coordinates, thus providing expressions that
do not involve normal vectors.

It will be convenient to introduce yet another multi-index; namely, we
let $\alpha=i_1\ldots i_{p-1}$ consist of $p-1$ indices all taking
values between $1$ and $m$.

\begin{proposition}\label{prop:normalvectors}
  For any value of the multi-index $\alpha$, the vector
  \begin{align}\label{eq:Zdef}
    Z_{\alpha}=\frac{1}{\gamma\paraa{n!\sqrt{(p-1)!}}}\gb^{ij}\eps_{jk_1\cdots k_n\alpha}\{x^{k_1},\ldots,x^{k_n}\}\d_i,
  \end{align}
  where $\eps_{i_1\cdots i_m}$ is the Levi-Civita tensor of $M$, is
  normal to $\TSigma$, i.e. $\gb(Z_{\alpha},e_a)=0$ for
  $a=1,2,\ldots,n$. For hypersurfaces ($p=1$), equation (\ref{eq:Zdef})
  defines a unique normal vector of unit length.
\end{proposition}

\begin{proof}
  To prove that $Z_\alpha$ are normal vectors, one simply notes that
  \begin{align*}
    \gamma\paraa{n!\sqrt{(p-1)!}}\gb(Z_\alpha,e_a) &= 
    \frac{1}{\rho}\eps^{a_1\cdots a_n}\eps_{jk_1\cdots k_n\alpha}\paraa{\d_ax^j}\paraa{\d_{a_1}x^{k_1}}\cdots\paraa{\d_{a_n}x^{k_n}}=0,
  \end{align*}
  since the $n+1$ indices $a,a_1,\ldots,a_n$ can only take on $n$
  different values and since
  $(\d_ax^j)(\d_{a_1}x^{k_1})\cdots(\d_{a_n}x^{k_n})$ is contracted
  with $\eps_{jk_1\cdots k_n\alpha}$ which is completely antisymmetric
  in $j,k_1,\ldots,k_n$. Let us now calculate $|Z|^2\equiv\gb(Z,Z)$
  when $p=1$. Using that\footnote{In our convention, no combinatorial factor is included in the antisymmetrization; for instance,  
    $\delta^{[i}_{[k}\delta^{j]}_{l]}=\delta^i_k\delta^j_l-\delta^i_l\delta^j_k$.}
  \begin{align*}
    \eps_{ik_1\cdots k_n}\eps^{il_1\cdots l_n} = \delta^{[l_1}_{[k_1}\cdots\delta^{l_n]}_{k_n]}
  \end{align*}
  one obtains
  \begin{align*}
    |Z|^2 &= \frac{1}{\gamma^2n!^2}\gb_{l_1l_1'}\cdots\gb_{l_nl_n'}
    \eps_{ik_1\cdots k_n}\eps^{il_1\cdots l_n}
    \{x^{k_1},\ldots,x^{k_n}\}\{x^{l_1'},\ldots,x^{l_n'}\}\\
    &=\frac{1}{\gamma^2n!^2}\gb_{l_1l_1'}\cdots\gb_{l_nl_n'}
    \delta^{[l_1}_{[k_1}\cdots\delta^{l_n]}_{k_n]}
    \{x^{k_1},\ldots,x^{k_n}\}\{x^{l_1'},\ldots,x^{l_n'}\}\\
    &=\frac{1}{\gamma^2n!}\{x^{l_1},\ldots,x^{l_n}\}
    \gb_{l_1l_1'}\cdots\gb_{l_nl_n'}\{x^{l_1'},\ldots,x^{l_n'}\}\\
    &=\frac{1}{\gamma^2n!}(n-1)!\Tr\P^2 = \frac{1}{\gamma^2n!}(n-1)!n\gamma^2=1,
  \end{align*}
  which proves that $Z$ has unit length.
\end{proof}

\noindent If the codimension is greater than one, $Z_\alpha$ defines
more than $p$ non-zero normal vectors that do not in general fulfill any
orthonormality conditions. In principle, one can now apply the
Gram-Schmidt orthonormalization procedure to obtain a set of $p$
orthonormal vectors. However, it turns out that one can use $Z_\alpha$
to construct another set of normal vectors, avoiding explicit use of
the Gram-Schmidt procedure; namely, introduce
\begin{align*}
  \Z_{\alpha}^{\beta} = \gb(Z_{\alpha},Z^\beta),
\end{align*}
and consider it as a matrix over multi-indices $\alpha$ and
$\beta$. As such, the matrix is symmetric (with respect to
$\gb_{\alpha\beta}\equiv \gb_{i_1j_1}\cdots\gb_{i_{p-1}j_{p-1}}$) and
we let $E_\alpha,\mu_\alpha$ denote orthonormal eigenvectors
(i.e. $\gb_{\delta\sigma}E_\alpha^\delta
E_\beta^\sigma=\delta_{\alpha\beta}$) and their corresponding
eigenvalues. Using these eigenvectors to define
\begin{align*}
  \Nh_\alpha = E^{\beta}_\alpha Z_\beta
\end{align*}
one finds that
$\gb(\Nh_\alpha,\Nh_\beta)=\mu_\alpha\delta_{\alpha\beta}$, i.e. the
vectors are orthogonal.

\begin{proposition}\label{prop:Zprojection}
  For $\Z_\alpha^\beta=\gb_{ij}Z^i_\alpha Z^{j\beta}$ it holds that
  \begin{align}
    \Z_\alpha^\delta\Z_\delta^\beta = \Z_\alpha^\beta\label{eq:Zidempot}\\
    \Z_\alpha^\alpha = p.\label{eq:Ztrace}
  \end{align}
\end{proposition}

\begin{proof}
  Both statements can be easily proven once one has the following result
  \begin{align*}
    Z^i_\alpha Z^{j\alpha} = \gb^{ij}-\frac{1}{\gamma^2}\paraa{\P^2}^{ij},
  \end{align*}
  which is obtained by using that
  \begin{align*}
    \eps_{kk_1\cdots k_n\alpha}\eps^{ll_1\cdots l_n\alpha} = 
    (p-1)!\parab{\delta^{[l}_{[k}\delta^{l_1}_{k_1}\cdots\delta^{l_n]}_{k_n]}}.
  \end{align*}
  Formula
  (\ref{eq:Ztrace}) is now immediate, and to obtain
  (\ref{eq:Zidempot}) one notes that since $\Z_\alpha\in\TSigma^\perp$
  it holds that $\P^2(\Z_\alpha)=0$, due to the fact that $\P^2$ is
  proportional to the projection onto $\TSigma$.
\end{proof}

\noindent From Proposition \ref{prop:Zprojection} it follows that an
eigenvalue of $\Z$ is either 0 or 1, which implies that $\Nh_\alpha=0$
or $\gb(\Nh_\alpha,\Nh_\alpha)=1$, and that the number of non-zero
vectors is $\Tr\Z = \Z_\alpha^\alpha=p$. Hence, the $p$ non-zero
vectors among $\Nh_\alpha$ constitute an orthonormal basis of
$\TSigma^\perp$, and it follows that one can replace any sum over
normal vectors $N_A$ by a sum over the multi-index of $\Nh_\alpha$. As
an example, let us work out some explicit expressions in the case when
$M=\reals^m$.

\begin{proposition}\label{prop:BARm}
  Assume that $M=\reals^m$ and that all repeated indices are summed
  over. For any $X\in\TSigma$ one has
  \begin{align*}
    &\sum_{A=1}^p\paraa{\Tr\B_A}\B_A(X)^i
    =\frac{\eps_{jj'K\alpha}\eps_{klL\alpha}}{\gamma^2c(n,p)^2} 
    \{x^j,\xv^I\}\{\xv^I,\{x^{j'},\xv^K\}\}
    \{x^i,\xv^J\}\{\xv^J,\{x^l,\xv^L\}\}X^k\\
    &\sum_{A=1}^p\B_A^2(X)^i
    =\frac{\eps_{jj'K\alpha}\eps_{klL\alpha}}{\gamma^2c(n,p)^2} 
    \{x^i,\xv^I\}\{\xv^I,\{x^{j'},\xv^K\}\}
    \{x^j,\xv^J\}\{\xv^J,\{x^l,\xv^L\}\}X^k\\
    &\sum_{A=1}^p\paraa{\Tr\B_A}N_A
    =\frac{(-1)^n}{n}\frac{\eps_{ikK\alpha}\eps_{jlL\alpha}}{\gamma^2c(n,p)^2}
    \{x^i,\xv^I\}\{\xv^I,\{x^k,\xv^K\}\}\{x^l,\xv^L\}\d_j
  \end{align*}
  where
  \begin{align*}
    c(n,p)=n!(n-1)!\sqrt{(p-1)!}.
  \end{align*}
\end{proposition}

\begin{proof}
  Let us prove the formula involving $(\Tr\B_A)\B_A(X)$; the other formulas
  can be proven analogously. For $\reals^m$ one has
  \begin{align*} 
    \sum_A\paraa{\Tr\B_A}\B_A(X)^i=
    \frac{1}{(n-1)!^2}\sum_A\{x^j,\xv^I\}\{\xv^I,n_A^j\}\{x^i,\xv^J\}\{\xv^J,n_A^k\}X^k,
  \end{align*}
  and since the non-zero vectors in the set $\{\Nh_\alpha\}$ consist
  of exactly $p$ orthonormal vectors one can write
  \begin{align*} 
    \sum_A\paraa{\Tr\B_A}\B_A(X)^i&=
    \frac{1}{(n-1)!^2}\sum_\alpha\{x^j,\xv^I\}\{\xv^I,\Nh_\alpha^j\}\{x^i,\xv^J\}\{\xv^J,\Nh_\alpha^k\}X^k\\
    & =\frac{1}{(n-1)!^2}\sum_\alpha\{x^j,\xv^I\}\{\xv^I,E_\alpha^\beta Z_\beta^j\}\{x^i,\xv^J\}\{\xv^J,E_\alpha^\epsilon Z_\epsilon^k\}X^k.
  \end{align*}
  Now, one notes that
  \begin{align*}
    &\sum_j\{x^j,\xv^I\}\{\xv^I,E_\alpha^\beta Z_\beta^j\}
    =\sum_j E_\alpha^\beta\{x^j,\xv^I\}\{\xv^I,Z_\beta^j\}\\
    &\sum_i\{x^i,\xv^J\}\{\xv^J,E_\alpha^\epsilon Z_\epsilon^k\}X^k
    =\sum_iE_\alpha^\epsilon\{x^j,\xv^I\}\{\xv^I,Z_\epsilon^k\}X^k
  \end{align*}
  since the terms with $Z_\beta^j$ and $Z_\epsilon^k$ outside the
  Poisson bracket vanish due to the appearance of a scalar product
  with a tangent vector. Thus, one obtains
  \begin{align*}
    \sum_A\paraa{\Tr\B_A}\B_A(X)^i&=
    \frac{1}{(n-1)!^2}\sum_\alpha E_\alpha^\beta E_\alpha^\epsilon\{x^j,\xv^I\}\{\xv^I,Z_\beta^j\}\{x^i,\xv^J\}\{\xv^J,Z_\epsilon^k\}X^k,
  \end{align*}
  and since $\sum_\alpha E_\alpha^\beta E_\alpha^\epsilon=\delta^{\beta\epsilon}$ the result is
  \begin{align*}
    \sum_A\paraa{\Tr\B_A}\B_A(X)^i&=
    \frac{1}{(n-1)!^2}\sum_\alpha\{x^j,\xv^I\}\{\xv^I,Z_\alpha^j\}\{x^i,\xv^J\}\{\xv^J,Z_\alpha^k\}X^k,
  \end{align*}
  from which the statement follows by inserting the definition of $Z_\alpha$.
\end{proof}

\noindent For hypersurfaces in $\reals^{n+1}$, the ``Theorema Egregium'' states
that the determinant of the Weingarten map, i.e the ``Gaussian
curvature'', is an invariant (up to a sign when $\Sigma$ is
odd-dimensional) under isometries (this is in fact also true for
hypersurfaces in a manifold of constant sectional curvature). From
Proposition \ref{prop:TrPBST} we know that one can express $\det W_A$
in terms of $\Tr\S_A\T_A$. 
\begin{proposition}
  Let $\Sigma$ be a hypersurface in $\reals^{n+1}$ and let $W$ denote
  the Weingarten map with respect to the unit normal
  \begin{align*}
    Z = \frac{1}{\gamma n!}\gb^{ij}\eps_{jkK}\{x^k,\xv^K\}.
  \end{align*}
  Then one can write $\det W$ as 
  \begin{align*}
    \det W = -\frac{1}{\gamma(\gamma n!)^{n+1}}&\sum
    \eps_{ilL}\eps_{j_1k_1K_1}\cdots\eps_{j_{n-1}k_{n-1}K_{n-1}}\\
    &\times\{x^i,\{x^{k_1},\xv^{K_1}\},\ldots,\{x^{k_{n-1}},\xv^{K_{n-1}}\}\}
    \{\xv^J,\{x^l,\xv^L\}\}.
  \end{align*}
\end{proposition}

\noindent In fact, one can express all the elementary symmetric
functions of the principle curvatures in terms of Nambu brackets as
follows: The $k$'th elementary symmetric function of the eigenvalues of
$W$ is given as the coefficient of $t^k$ in $\det(W-t\mid)$. Since
$\B(X)=0$ for all $X\in\TSigma^\perp$ and $\B(X)=\gamma^2W(X)$ for all
$X\in\TSigma$, it holds that
\begin{align*}
  -t\det(W-t\mid_n) = \det(\gamma^{-2}\B-t\mid_{n+1})
  =\frac{1}{\gamma^{2(n+1)}}\det(\B-t\gamma^2\mid_{n+1})
\end{align*}
which implies that the $k$'th symmetric function is given by the
coefficient of $t^{k+1}$ in $-\det(\B-t\gamma^2\mid)\gamma^{2(n-k)}$.

\section{The Codazzi-Mainardi equations}\label{sec:CodazziMainardi}

\noindent When studying the geometry of embedded manifolds, the Codazzi-Mainardi
equations are very useful. In this section we reformulate these equations
in terms of Nambu brackets.

The Codazzi-Mainardi equations express the normal component
of $\Rb(X,Y)Z$ in terms of the second fundamental forms; namely
\begin{align}\label{eq:CMh}
  \begin{split}
    \gb\paraa{&\Rb(X,Y)Z,N_A} = \paraa{\nabla_Xh_A}(Y,Z) - \paraa{\nabla_Yh_A}(X,Z)\\
    &+\sum_{A=1}^p\bracketb{\gb(D_XN_B,N_A)h_B(Y,Z)-\gb(D_YN_B,N_A)h_B(X,Z)},
  \end{split}
\end{align}
for $X,Y,Z\in\TSigma$ and $A=1,\ldots,p$. Defining 
\begin{align}
  \begin{split}
    \W_A&(X,Y) = \paraa{\nabla_XW_A}(Y)-\paraa{\nabla_Y W_A}(X)\\
    &+\sum_{B=1}^p\bracketb{\gb(D_XN_B,N_A)W_B(Y)-\gb(D_YN_B,N_A)W_B(X)}
  \end{split}
\end{align}
one can rewrite the Codazzi-Mainardi equations as follows.
\begin{proposition}\label{prop:CMWPi}
  Let $\Pi$ denote the projection onto $\TSigma^\perp$.  Then the
  Codazzi-Mainardi equations are equivalent to
  \begin{align}\label{eq:CMCA}
    \W_A(X,Y) = -(\mid-\Pi)\paraa{\Rb(X,Y)N_A}
  \end{align}
  for $X,Y\in\TSigma$ and $A=1,\ldots,p$.
\end{proposition}

\begin{proof}
  Since $h_A(X,Y)=\gb(W_A(X),Y)$ (by Weingarten's equation) one can
  rewrite (\ref{eq:CMh}) as
  \begin{align}
    \gb\paraa{\W_A(X,Y),Z} = \gb\paraa{\Rb(X,Y)Z,N_A},
  \end{align}
  and since $\gb(\Rb(X,Y)Z,N_A) = -\gb(\Rb(X,Y)N_A,Z)$ this becomes
  \begin{align}
    \gb\paraa{\W_A(X,Y)+\Rb(X,Y)N_A,Z} = 0.
  \end{align}
  That this holds for all $Z\in\TSigma$ is equivalent to saying that 
  \begin{align}
    (\mid-\Pi)\paraa{\W_A(X,Y)+\Rb(X,Y)N_A} = 0,
  \end{align}
  from which (\ref{eq:CMCA}) follows since $\W_A(X,Y)\in\TSigma$.
\end{proof}

\noindent Note that since $\gamma^{-2}\P^2$ is the projection onto
$\TSigma$ one can write (\ref{eq:CMCA}) as
\begin{align}\label{eq:CMPgamma}
  \gamma^2\W_A(X,Y) = -\P^2\paraa{\Rb(X,Y)N_A}.
\end{align}
\noindent Since both $W_A$ and $D_X$ can be expressed in terms of
$\B_A$, one obtains the following expression for $\W_A$:
\begin{proposition}
  For $X,Y\in\TSigma$ one has
  \begin{align*}
    \gamma^2\W_A(X,Y) = &\paraa{\nablab_X\B_A}(Y)-\paraa{\nablab_Y\B_A}(X)\\
    &-\frac{1}{\gamma^2}\bracketb{\paraa{\nabla_X\gamma^2}\B_A(Y)-\paraa{\nabla_Y\gamma^2}\B_A(X)}\\
    &+\frac{1}{\gamma^2}\sum_{B=1}^p\bracketb{\gb\paraa{\B_A(N_B),X}\B_B(Y)-\gb\paraa{\B_A(N_B),Y}\B_B(X)}.
  \end{align*}
\end{proposition}

\noindent As the aim is to express the Codazzi-Mainardi equations in
terms of Nambu brackets, we will introduce maps $\C_A$ that
is defined in terms of $\W_A$ and can be written as expressions
involving Nambu brackets.

\begin{definition}
  The maps $\C_A:C^\infty(\Sigma)\times\cdots\times
  C^\infty(\Sigma)\to \TSigma$ are defined as
  \begin{align}
    \C_A(f_1,\ldots,f_{n-2}) = \frac{1}{2\rho}\eps^{aba_1\cdots a_{n-2}}
    \W_A(e_a,e_b)\paraa{\d_{a_1}f_1}\cdots\paraa{\d_{a_{n-2}}f_{n-2}}
  \end{align}
  for $A=1,\ldots,p$ and $n\geq 3$. When $n=2$, $\C_A$ is defined as
  \begin{align*}
    \C_A = \frac{1}{2\rho}\eps^{ab}\W_A(e_a,e_b).
  \end{align*}
\end{definition}

\begin{proposition}\label{prop:CANPbracket}
  Let $\{g_1,g_2\}_f\equiv\{g_1,g_2,f_1,\ldots,f_{n-2}\}$. Then
  \begin{align*}
    \C_A(f_1,\ldots,&f_{n-2})^i = 
    \pb{\gamma^{-2}(\B_A)^i_k,x^k}_f
    +\frac{1}{\gamma^2}\pb{x^j,x^l}_f\bracketb{\Gammab^i_{jk}(\B_A)^k_l-(\B_A)^i_k\Gammab^k_{jl}}\\
    &-\frac{1}{\gamma^2}\sum_{B=1}^p\bracketb{
      \pb{n_A^k,x^l}_f(\B_B)^i_l+\Gammab^k_{lj}\pb{x^l,x^m}_fn_A^j(\B_B)^i_m
    }(n_B)_k.
  \end{align*}
\end{proposition}

\begin{remark}
  \noindent In case $\Sigma$ is a hypersurface, the expression for $\C\equiv \C_1$ simplifies to
  \begin{align*}
    \C(f_1,\ldots,f_{n-2})^i = 
    &\pb{\gamma^{-2}\B^i_k,x^k}_f
    +\frac{1}{\gamma^2}\pb{x^j,x^l}_f\bracketb{\Gammab^i_{jk}\B^k_l-\B^i_k\Gammab^k_{jl}},
  \end{align*}
  since $D_XN=0$.  
\end{remark}

\noindent It follows from Proposition \ref{prop:CMWPi} that we can
reformulate the Codazzi-Mainardi equations in terms of $\C_A$:

\begin{theorem}\label{thm:CMNambu}
  For all $f_1,\ldots,f_{n-2}\in C^\infty(\Sigma)$ it holds that
  \begin{align}\label{eq:CMNambu}
    \gamma^2\C_A(f_1,\ldots,f_{n-2}) = (\P^2)^{i}_j\bracketb{\{x^k,\Gammab^j_{kj'}\}_f-\pb{x^k,x^l}_f\Gammab^m_{lj'}\Gammab^j_{km}}n_A^{j'}\d_i,
  \end{align}
  for $A=1,\ldots,p$, where $\{g_1,g_2\}_f=\{g_1,g_2,f_1,\ldots,f_{n-2}\}$.
\end{theorem}

\begin{proof}
  As noted previously, one can write the Codazzi-Mainardi equations as
  \begin{align*}
    \gamma^2\W_A(X,Y) = -\P^2\paraa{\Rb(X,Y)N_A}.
  \end{align*}
  That the above equation holds for all $X,Y\in\TSigma$ is equivalent to saying that
  \begin{align*}
    \gamma^2\frac{1}{2\rho}\eps^{aba_1\cdots a_{n-2}}\W_A(e_a,e_b)
    =-\frac{1}{2\rho}\eps^{aba_1\cdots a_{n-2}}\P^2\paraa{\Rb(e_a,e_b)N_A}
  \end{align*}
  for all values of $a_1,\ldots,a_{n-2}\in\{1,\ldots,n\}$; furthermore, this is equivalent to
  \begin{align*}
    \gamma^2\C_A(f_1,\ldots,f_{n-2}) = 
    -\frac{1}{2\rho}\eps^{aba_1\cdots a_{n-2}}\P^2\paraa{\Rb(e_a,e_b)N_A}
    (\d_{a_1}f_1)\cdots(\d_{a_{n-2}}f_{n-2})
  \end{align*}
  for all $f_1,\ldots,f_{n-2}\in C^\infty(\Sigma)$. It is now straightforward to show that
\begin{align*}
    -\frac{1}{2\rho}\eps^{aba_1\cdots a_{n-1}}&\paraa{\Rb(e_a,e_b)N_A}^i(\d_{a_1}f_1)\cdots(\d_{a_{n-2}}f_{n-2})\\
    &=\parab{\{x^k,\Gammab^i_{kj}\}_f-\pb{x^k,x^l}_f\Gammab^m_{lj}\Gammab^i_{km}}n_A^{j},
  \end{align*}
  which proves the statement.
\end{proof}

\noindent If $M$ is a space of constant curvature (in which case $\gb(\Rb(X,Y)Z,N_A)=0$), then Theorem \ref{thm:CMNambu} states that
\begin{align}
  \C_A(f_1,\ldots,f_{n-2}) = 0
\end{align}
for all $f_1,\ldots,f_{n-2}\in C^\infty(\Sigma)$. Furthermore, if $M=\reals^m$, then (\ref{eq:CMNambu}) becomes
\begin{align}
  \gamma^2\pb{\gamma^{-2}(\B_A)^i_k,x^k}_f-\sum_{B=1}^p\bracketb{
    \pb{n_A^k,x^l}_f(\B_B)^i_l}(n_B)_k = 0.
\end{align}

\section{Embedded surfaces}\label{sec:surfaces}

\noindent Let us now turn to the special case when $\Sigma$ is a
surface. For surfaces, the tensors $\P$, $\S_A$ and $\T_A$ are themselves
maps from $TM$ to $TM$, and $\S_A$ coincides with
$\T_A$. Moreover, since the second fundamental forms can be considered
as $2\times 2$ matrices, one has the identity
\begin{align*}
  2\det W_A = \paraa{\tr W_A}^2-\tr W_A^2,
\end{align*}
which implies that the scalar curvature can be written as
\begin{align*}
  R &= g^{ac}g^{bd}\gb\paraa{\Rb(e_c,e_d)e_b,e_a}
  + 2\sum_{A=1}^p\det W_A\\
  &= 2\frac{\gb\paraa{\Rb(e_1,e_2)e_2,e_1}}{g}+ 2\sum_{A=1}^p\det W_A.
\end{align*}
Thus, defining the Gaussian curvature $K$ to be one half of the above
expression (which also coincides with the sectional curvature), one obtains
\begin{align}
  K = \frac{\gb\paraa{\Rb(e_1,e_2)e_2,e_1}}{g}-\frac{1}{2\gamma^2}\sum_{A=1}^p\Tr\S_A^2,
\end{align}
which in the case when $M=\reals^m$ becomes
\begin{align}
  K =-\frac{1}{2\gamma^2}\sum_{A=1}^p\sum_{i,j=1}^m\{x^i,n_A^j\}\{x^j,n_A^i\},
\end{align}
and by using the normal vectors $Z_\alpha$ the expression for $K$ can
be written as
\begin{align}
  K &= -\frac{1}{8\gamma^4(p-1)!}
  \sum\eps_{jklI}\eps_{imnI}\{x^i,\{x^k,x^l\}\}\{x^j,\{x^m,x^n\}\}.
\end{align}

\noindent To every Riemannian metric on $\Sigma$ one can associate an
almost complex structure $\J$ through the formula
\begin{equation*}
  \J(X) = \frac{1}{\sqrt{g}}\eps^{ac}g_{cb}X^be_a,
\end{equation*}
and since on a two dimensional manifold any almost complex structure
is integrable, $\J$ is a complex structure on $\Sigma$. For $X\in TM$ one has
\begin{align}
  \P(X) = -\frac{1}{\gamma\sqrt{g}}\gb\paraa{X,e_a}\eps^{ab}e_b,
\end{align}
and it follows that one can express the complex structure in terms of $\P$.

\begin{theorem}\label{thm:complexstructure}
  Defining $\JM(X)=\gamma\P(X)$ for all $X\in TM$ it holds that
  $\J_M(Y)=\J(Y)$ for all $Y\in\TSigma$. That is, $\gamma\P$ defines a
  complex structure on $\TSigma$.
\end{theorem}

\noindent Let us now turn to the Codazzi-Mainardi equations for
surfaces. In this case, the map $\C_A$ becomes a tangent vector and 
one can easily see in Proposition \ref{prop:CANPbracket} that the sum
in the expression for $\C_A$ can be written in a slightly more compact form, namely
\begin{align*}
  \C_A = &\pb{\gamma^{-2}(\B_A)^i_k,x^k}\d_i
  +\frac{1}{\gamma^2}\pb{x^j,x^l}\bracketb{\Gammab^i_{jk}(\B_A)^k_l-(\B_A)^i_k\Gammab^k_{jl}}\\
  &\qquad+\frac{1}{\gamma^2}\sum_{B=1}^p\B_B\S_A(N_B).
\end{align*}
\noindent Thus, for surfaces embedded in $\reals^m$ the Codazzi-Mainardi equations become
\begin{align*}
  \sum_{j,k=1}^m\pb{\gamma^{-2}\{x^i,x^j\}\{x^j,n_A^k\},x^k}\d_i+\frac{1}{\gamma^2}\sum_{B=1}^p\B_B\S_A(N_B)=0,
\end{align*}
and in $\reals^3$ one has
\begin{align}
  \sum_{j,k=1}^3\pba{\gamma^{-2}\{x^i,x^j\}\{x^j,n^k\},x^k} = 0.
\end{align}

\noindent Let us note that one can rewrite these equations using the following result:

\begin{proposition}
  For $M=\reals^m$ and $i=1,\ldots,m$ it holds that
  \begin{align}
    \sum_{j,k=1}^m\pba{f\{x^i,x^j\}\{x^j,n^k\},x^k} = 
    \sum_{j,k=1}^m\pba{f\{x^i,x^j\}\{x^j,x^k\},n^k}
  \end{align}
  for any normal vector $N=n^i\d_i$ and any $f\in C^\infty(\Sigma)$.
\end{proposition}

\begin{proof}
  We start by recalling that for any $g\in C^\infty(\Sigma)$ it holds
  that $\sum_{i=1}^m\{g,x^i\}n^i=0$, since it involves the scalar product
  $\gb(e_a,N)$. Moreover, one also has
  \begin{align*}
    \sum_{k=1}^m\{x^k,n^k\} &= \sum_{k=1}^m\frac{1}{\rho}\eps^{ab}(\d_ax^k)(\d_bn^k)
    =\sum_{k=1}^m\frac{1}{\rho}\eps^{ab}\parab{\d_b\paraa{n^k\d_ax^k}-n^k\d^2_{ab}x^k}\\
    &=-\sum_{k=1}^m\frac{1}{\rho}\eps^{ab}n^k\d^2_{ab}x^k=0,
  \end{align*}
  which implies that $\sum_{k=1}^m\{x^k,gn^k\}=0$ for all $g\in C^\infty(\Sigma)$.
  By using the above identities together with the Jacobi identity, one
  obtains
  \begin{align*}
    \pba{f\{x^i,x^j\}\{x^j,n^k\},x^k} &= 
    f\{x^i,x^j\}\pba{\{x^j,n^k\},x^k}+\{x^j,n^k\}\pba{f\{x^i,x^j\},x^k}\\
    &= -f\{x^i,x^j\}\pba{\{x^k,x^j\},n^k}-n^k\pba{x^j,\{f\{x^i,x^j\},x^k\}}\\
    &= -f\{x^i,x^j\}\pba{\{x^k,x^j\},n^k} + n^k\pba{f\{x^i,x^j\},\{x^k,x^j\}}\\
    &= -f\{x^i,x^j\}\pba{\{x^k,x^j\},n^k} - \{x^k,x^j\}\pba{f\{x^i,x^j\},n^k}\\
    &= \pba{f\{x^i,x^j\}\{x^j,x^k\},n^k}.\qedhere
  \end{align*}
\end{proof}

\noindent Hence, one can rewrite the Codazzi-Mainardi equations for a surface in $\reals^3$ as
\begin{align}\label{eq:CMR3P}
  \sum_{j,k=1}^3\pba{\gamma^{-2}(\P^2)^{ik},n^k} = 0,
\end{align}
and it is straight-forward to show that 
\begin{align*}
  \sum_{i,j,k=1}^3\paraa{\d_cx^i}\pba{\gamma^{-2}(\P^2)^{ik},n^k} = \frac{1}{\rho}\eps^{ab}\nabla_ah_{bc},
\end{align*}
thus reproducing the classical form of the Codazzi-Mainardi equations.

Is it possible to verify (\ref{eq:CMR3P}) directly using only Poisson
algebraic manipulations? It turns out that that the Codazzi-Mainardi
equations in $\reals^3$ is an identity for arbitrary Poisson algebras,
if one assumes that a normal vector is given by
$\frac{1}{2\gamma}\eps_{ijk}\{x^j,x^k\}\d_i$.
\begin{proposition}
  Let $\{\cdot,\cdot\}$ be an arbitrary Poisson structure on
  $C^\infty(\Sigma)$. Given $x^1,x^2,x^3\in C^\infty(\Sigma)$
  it holds that
  \begin{align*}
    \sum_{j,k,l,n=1}^3\frac{1}{2}\eps_{kln}\pba{\gamma^{-2}\{x^i,x^j\}\{x^j,x^k\},\gamma^{-1}\{x^l,x^n\}}=0
  \end{align*}
  for $i=1,2,3$, where
  \begin{align*}
    \gamma^2 = \{x^1,x^2\}^2+\{x^2,x^3\}^2+\{x^3,x^1\}^2.
  \end{align*}
\end{proposition}

\begin{proof}
  Let $u,v,w$ be a cyclic permutation of $1,2,3$. In the following we
  do not sum over repeated indices $u,v,w$. Denoting by $\text{CM}^i$
  the $i$'th component of the Codazzi-Mainardi equation, one has
  \begin{align*}
    &\text{CM}^u = -\pba{\gamma^{-2}\paraa{\{x^u,x^v\}^2+\{x^w,x^u\}^2},\gamma^{-1}\{x^v,x^w\}}\\
    &\quad+\pba{\gamma^{-2}\{x^u,x^v\}\{x^v,x^w\},\gamma^{-1}\{x^u,x^v\}}
    +\pba{\gamma^{-2}\{x^u,x^w\}\{x^w,x^v\},\gamma^{-1}\{x^w,x^u\}}\\
    &\quad= -\pba{1-\gamma^{-2}\{x^v,x^w\}^2,\gamma^{-1}\{x^v,x^w\}}
    +\gamma^{-1}\{x^u,x^v\}\pba{\gamma^{-1}\{x^v,x^w\},\gamma^{-1}\{x^u,x^v\}}\\
    &\quad+\gamma^{-1}\{x^u,x^w\}\pba{\gamma^{-1}\{x^w,x^v\},\gamma^{-1}\{x^w,x^u\}}\\
    &\quad= \frac{1}{2}\pba{\gamma^{-1}\{x^v,x^w\},\gamma^{-2}\paraa{\gamma^2-\{x^v,x^w\}^2}}=0.\qedhere
  \end{align*}
\end{proof}

\noindent Let us end by noting that these results generalize to
arbitrary hypersurfaces in $\reals^{n+1}$. Namely, 
\begin{align*}
  &\{\gamma^{-2}\pba{x^i,\xv^J\}\{\xv^J,n^k\},x^k}_f
  =\{\gamma^{-2}\pba{x^i,\xv^J\}\{\xv^J,x^k\},n^k}_f,\\
  &(\d_cx^i)\pba{\gamma^{-2}\paraa{\P^2}^{ik},n^k}_f=
  -\frac{1}{\rho}\eps^{aba_1\cdots a_{n-2}}\paraa{\nabla_ah_{bc}}\paraa{\d_{a_1}f_1}\cdots\paraa{\d_{a_{n-2}}f_{n-2}},
\end{align*}
and
\begin{align*}
  \eps_{klL}\pba{\gamma^{-2}\{x^i,\xv^J\}\{\xv^J,x^k\},\gamma^{-1}\{x^l,\xv^L\}}_f=0
\end{align*}
for arbitrary $x^1,\ldots,x^{n+1}\in C^\infty(\Sigma)$.

\bibliographystyle{alpha}
\bibliography{nsurface}

\end{document}